\documentclass[12pt]{amsart}
\usepackage{tgpagella}
\usepackage{euler}
\usepackage[T1]{fontenc}
\usepackage{amsmath, amssymb}
\usepackage[hidelinks]{hyperref}
\usepackage[english]{babel}
\usepackage{mathrsfs}
\usepackage{eucal}
\usepackage[all]{xy}

\newtheorem{thm}{Theorem}[section]
\newtheorem{defn}[thm]{Definition}

\newtheorem{prop}[thm]{Proposition}
\newtheorem{cor}[thm]{Corollary}

\newtheorem{fact}[thm]{Fact}

\newcommand{\dminus}{ 
\buildrel\textstyle\ .\over{\hbox{ 
\vrule height3pt depth0pt width0pt}{\smash-} 
}}

\def \u{\mathbb U}
\newcommand{\Conf}{\mathrm{Conf}}
\newcommand{\thesentence}{\varphi_{\geq\frac{1}{2}}}

\makeatletter

\def\indsym#1#2{%
  \setbox0=\hbox{$\m@th#1x$}%
  \kern\wd0%
  \hbox to 0pt{\hss$\m@th#1\mid$\hbox to 0pt{$\m@th#1^{#2}$}\hss}%
  \lower.9\ht0\hbox to 0pt{\hss$\m@th#1\smile$\hss}%
  \kern\wd0}
\newcommand{\ind}[1][]{\mathop{\mathpalette\indsym{#1}}}

\def\nindsym#1#2{%
  \setbox0=\hbox{$\m@th#1x$}%
  \kern\wd0%
  \hbox to 0pt{\hss$\m@th#1\not$\kern1.4\wd0\hss}
  \hbox to 0pt{\hss$\m@th#1\mid$\hbox to 0pt{$\m@th#1^{\,#2}$}\hss}%
  \lower.9\ht0\hbox to 0pt{\hss$\m@th#1\smile$\hss}%
  \kern\wd0}

\def \as{\operatorname{AS}}
\title{The almost sure theory of finite metric spaces}
\author{Isaac Goldbring, Bradd Hart, and Alex Kruckman}
\thanks{I. Goldbring was partially supported by NSF CAREER grant DMS-1349399. B. Hart was supported by an NSERC Discovery Grant.}

\address{Department of Mathematics\\University of California, Irvine, 340 Rowland Hall (Bldg.\# 400),
Irvine, CA 92697-3875}
\email{isaac@math.uci.edu}
\urladdr{http://www.math.uci.edu/~isaac}

\address{Department of Mathematics and Statistics, McMaster University, 1280 Main St., Hamilton ON, Canada L8S 4K1}
\email{hartb@mcmaster.ca}
\urladdr{http://ms.mcmaster.ca/~bradd/}

\address{Department of Mathematics and Computer Science, Wesleyan University, Science Tower 655, 265 Church St., Middletown, CT 06459-0128}
\email{akruckman@wesleyan.edu}
\urladdr{https://akruckman.faculty.wesleyan.edu/}

\begin{document}

\begin{abstract}
We establish an approximate zero-one law for sentences of continuous logic over finite metric spaces of diameter at most $1$. More precisely, we axiomatize a complete metric theory $T_{\as}$ such that, given any sentence $\sigma$ in the language of pure metric spaces and any $\epsilon>0$, the probability that the difference of the value of $\sigma$ in a random metric space of size $n$ and the value of $\sigma$ in any model of $T_{\as}$ is less than  $\epsilon$ approaches $1$ as $n$ approaches infinity.  We also establish some model-theoretic properties of the theory $T_{\as}$.
\end{abstract}
\maketitle

\section{Introduction}

The classical zero-one law for finite graphs~\cite{BH,Fagin} says that for any property of graphs which is definable by a sentence of first-order logic, either almost all finite graphs satisfy the property, or almost none of them do. More precisely, let $\mathcal{G}_n$ be the set of all graphs with vertex set $\{1,\dots,n\}$, and let $\mu_n$ be the uniform probability measure on the finite set $\mathcal{G}_n$. For any first-order sentence $\varphi$ in the language of graphs, write $[\varphi]_n$ for the set of graphs in $\mathcal{G}_n$ satisfying $\varphi$. Then: \begin{equation*}\label{eqn:graphs}\lim_{n\to\infty} \mu_n([\varphi]_n) = 0\text{ or }1.\end{equation*}

The proof can be summarized as follows: First, one writes down a set of \emph{extension axioms} and proves (by an enumerative argument) that any finitely many of these axioms hold for almost all finite graphs. Next, one proves (by a logical argument) that these extension axioms axiomatize a complete first-order theory, the \emph{almost sure theory} $T_{\as}$ of finite graphs. Completeness means that for any sentence $\varphi$, either $\varphi$ or its negation is in $T_{\as}$. It follows by the compactness theorem for first-order logic that either $\varphi$ or its negation is entailed by finitely many extension axioms, and hence holds in almost all finite graphs. 

Zero-one laws have been an active area of research since the 1970s. Generalizations have proceeded in three main directions: By considering more complicated classes of structures, by replacing the uniform measures $\mu_n$ with more general random models,  and by generalizing to more expressive logics than first-order. See~\cite{Compton} and~\cite{Winkler} for comprehensive surveys. 

In this paper, we initiate the study of zero-one laws for metric structures in continuous logic. Continuous logic is a generalization of first-order logic in which sentences define properties of metric spaces, possibly equipped with extra structure. In continuous logic, the standard truth values $\{\top,\bot\}$ are replaced by the unit interval $[0,1]$, and the primitive equality ($=$) is replaced by a metric ($d$). Since $x=y$ in a metric space if and only if $d(x,y) = 0$, we think of $0$ as playing the role of $\top$ (``true'') and $1$ as playing the role of $\bot$ (``false''). The canonical reference for continuous logic is~\cite{mtfms}. 

Throughout the paper, $L$ denotes the ``empty'' metric language, that is, the metric language consisting solely of the metric symbol $d$.  The words ``formula'' and ``sentence'' will be used as abbreviations for ``$L$-formula'' and ``$L$-sentence'' respectively.  All metric spaces will have diameter at most $1$ and will be viewed as $L$-structures.

To formulate a continuous zero-one law for finite metric spaces, we need to make two adjustments to the classical statement. First, since there are infinitely many metric spaces on the set $\{1,\dots,n\}$, we cannot make sense of a ``random finite metric space'' by using a naive uniform measure. Instead, we encode a metric space $X$ on $\{1,\dots,n\}$ by a point in $[0,1]^{\binom{n}{2}}$ (the sequence of distances between pairs of distinct points in $X$) and work with the probability measure $\nu_n$ which is the normalization of the Lebesgue measure on $[0,1]^{\binom{n}{2}}$ restricted to the set $M_n$ of points which encode metric spaces. 

Second, our zero-one law has an approximate character: instead of asking whether a sentence of continuous logic is true or false, it is more natural to ask whether its truth value lies in some subinterval of $[0,1]$. In Theorem~\ref{thm:main} below, we show that for any sentence $\sigma$, there is a truth value $r_\sigma\in [0,1]$ such that for almost all finite metric spaces $X$, the truth value of $\sigma$ in $X$ (denoted $\sigma^X$) is approximately $r_\sigma$. More precisely, for all $\epsilon>0$:
\begin{equation*}\label{eqn:metric}
\lim_{n\to \infty}\nu_n(\{X\in M_n \ : \ |\sigma^{X}-r_\sigma|<\epsilon\})=1.\end{equation*}

The proof of Theorem~\ref{thm:main} proceeds along the same lines as the proof of the classical zero-one law for finite graphs. In Section 2, we define ``approximate extension axioms'' and show that they hold almost surely in sufficiently large finite metric spaces.  In Section 3, we show that these extension axioms axiomatize a complete theory $T_{\as}$, the almost sure theory of finite metric spaces. Theorem~\ref{thm:main} follows by an approximate version of the compactness theorem for continuous logic (Fact~\ref{compactness}). The final section establishes further model-theoretic properties of the almost sure theory, including the fact that it has quantifier-elimination, has continuum many non-isomorphic separable models, and is unstable but is supersimple of U-rank $1$.

In Sections 2 and 3, we assume minimal familiarity with continuous logic as established in \cite[Sections 2-4]{mtfms}.  The final section will assume familiarity with more sophisticated model-theoretic notions.

We conclude this introduction with some remarks on pseudofiniteness and Fra\"iss\'e limits. In the case of the zero-one law for finite graphs, the almost-sure theory is equal to the complete theory of a particular countably infinite graph $G_R$: the random graph (also known as the Rado graph), which is the Fra\"iss\' e limit of the class of finite graphs. As a consequence, the random graph is \emph{pseudofinite}. This means that $G_R$ is elementarily equivalent to an ultraproduct of finite graphs; equivalently, every sentence which is true of $G_R$ has a finite model.

The class of finite metric spaces of diameter at most $1$ also has a Fra\"iss\' e limit: the \emph{Urysohn sphere $\u$}. In an earlier preprint, the first two authors claimed that the almost-sure theory of finite metric spaces $T_{\as}$ is equal to the complete theory of $\u$, and as a consequence, that $\u$ is pseudofinite.\footnote{In the continuous setting, a structure $M$ is pseudofinite if and only if it is elementarily equivalent to an ultraproduct of finite structures; equivalently, for any sentence $\sigma$ such that $\sigma^M = 0$ and for any $\epsilon>0$, there is some finite structure $A$ such that $\sigma^A<\epsilon$.} However, a serious flaw in the argument was discovered by the third author, and thus the question of the pseudofiniteness of the Urysohn sphere is still open.  

The difference between $T_{\as}$ and the complete theory of $\u$ rests on the fact that almost all finite metric spaces of diameter at most $1$ have all of their non-zero distances approximately in the interval $[\frac{1}{2},1]$. See Fact~\ref{CS} below for a more precise statement, which was proven by Kozma, Meyerovitch, Peled, and Samotij in \cite{KMPS}. A similar result in the discrete setting was proven by Mubayi and Terry in \cite{MT}. As a result, our axiomatization of $T_{\as}$ consists of extension axioms for metric spaces in which all distances are at least $\frac{1}{2}$, while the complete theory of $\u$ can be axiomatized by extension axioms for metric spaces with arbitrary distances in $[0,1]$.

It is worth drawing a parallel here with the zero-one laws for the classes $\mathcal{G}_n$ of $K_n$-free graphs ($n\geq 3$). A finite metric space can be viewed as a complete graph with edges labeled by $(0,1]$, omitting certain configurations (those edge-labeled triangles which would violate the triangle inequality). Requiring all non-zero distances to be at least $\frac{1}{2}$ ensures that there are no non-trivial instances of the triangle inequality, and hence is an efficient way of avoiding all forbidden configurations. From this perspective, the enumerative result from~\cite{KMPS} is analogous to the result of Erd\H{o}s, Kleitman, and Rothschild from~\cite{EKR} that almost all $K_n$-free graphs are $(n-1)$-partite, and the current paper is analogous to~\cite{KPR}, in which Kolaitis, Pr\"omel, and Rothschild proved that $\mathcal{G}_n$ has a zero-one law, with almost-sure theory axiomatized by extension axioms for $(n-1)$-partite graphs. 

The Fra\"{i}ss\'{e} limit of the class $\mathcal{G}_n$ is called the \emph{Henson graph} $H_n$, and the pseudofiniteness of the Henson graphs is a notorious open problem (see~\cite{Cherlin} for a detailed investigation into the case $n = 3$). It seems likely to us that resolving the question of pseudofiniteness of the Urysohn sphere will be difficult for similar reasons. 

The authors would like to thank Henry Towsner, Caroline Terry, and Gabriel Conant for helpful comments throughout the writing of this paper.  We would also like to thank the authors of \cite{KMPS} for access to their preprint.

\section{Approximate extension axioms}

Let $\mathcal C$ be the class of finite metric spaces in which the distance between any two distinct points lies in the interval $[\frac{1}{2},1]$.  Note that all such metric spaces are discrete, as any ball of radius $\frac{1}{4}$ consists just of its center.

Given a finite metric space $X=\{x_1,\ldots,x_n\}$, we let $\Conf_X(v_1,\ldots,v_n)$ denote the formula
\[
\max_{1\leq i < j\leq n} |d(x_i,x_j) - d(v_i,v_j)|.
\]

Here, ``$\Conf$'' is short for ``Configuration'' and reminds us that the formula is comparing the distance configurations between the points $x_1,\ldots,x_n$ and what is being plugged in for the tuple of variables $v_1,\ldots,v_n$, whence $\Conf_X(v_1,\ldots,v_n)$ is small if these two sets of points are configured in approximately the same manner.

We use the notation $X\sqsubset Y$ when $X$ is a finite metric space and $Y$ is a one-point extension of $X$, in which case the extra point is denoted by $y$.

Given $X\sqsubset Y$ with $X,Y\in \mathcal  C$, we let $\Psi^\epsilon_{X\sqsubset Y}$ denote the sentence

\[
\sup_{\bar v} \min \left\{ \epsilon \dminus \Conf_X (\vec v), \inf_w \Conf_Y(\vec v,w) \dminus \epsilon \right\}.
\]

Thus, $\Psi^\epsilon_{X\sqsubset Y}$ evaluates to $0$ in some metric space $A$ if and only if:  whenever $\bar a$ is tuple from $A$ such that $\Conf_X(\bar a)<\epsilon$ and $\delta>0$ is a positive real number, there is $b\in A$ such that $\Conf_Y(\bar a,b)<\epsilon+\delta$. That is, every subspace which approximately looks like $X$ has an extension to a subspace which approximately looks like $Y$. 

We identify $\bar{d}=(d_{ij}: 1 \leq i < j \leq n)\in [0,1]^{\binom{n}{2}}$ with the metric space on $\{1,\ldots,n\}$ with $d(i,j):=d_{ij}$.  In this manner, if $X\in \mathcal{C}$, we write $\Conf_X(\bar{d})$, with the interpretation that the appearance of $d(v_i,v_j)$ gets replaced with $d_{ij}$.  We perform a similar identification with $\Psi^\epsilon_{X\sqsubset Y}(\bar d)$.

Let $M_n\subseteq [0,1]^{\binom{n}{2}}$ denote the set of all metric spaces on $\{1,\ldots,n\}$ with values in $[0,1]$. We let $\lambda_n$ be Lebesgue measure on $[0,1]^{\binom{n}{2}}$, and we let $\nu_n$ be Lebesgue measure normalized to $M_n$, that is, 
\[
\nu_n(A)=\frac{\lambda_n(A)}{\lambda_n(M_n)}.
\]

\noindent The following is a less precise version of \cite[Theorem 1.3]{KMPS}.

\begin{fact}\label{CS}
There is a decreasing sequence $(\delta_n)$ from $(0,\frac{1}{2})$ which tends to 0 such that, setting $$D_n:=\left\{\bar d \in M_n \ : \ d_{ij} \geq \frac{1}{2}-\delta_n \text{ for all }1\leq i<j\leq n\right\},$$ we have $\lim_{n\to \infty}\nu_n(D_n)=1$.

\end{fact}

We let $\mu_n$ be Lebesgue measure normalized to $D_n$.

\begin{thm}\label{newkeytheorem}
For any $X_1 \sqsubset Y_1, \ldots, X_m \sqsubset Y_m$ from $\mathcal C$ and any $\epsilon > 0$, we have
\[
\lim_{n\to \infty}\mu_n\left(\left\{\bar d\in D_n \ : \ \max_{i=1,\ldots,m}\Psi_{X_i\sqsubset Y_i}^\epsilon(\bar d)=0\right\}\right)=1.
\]
\end{thm}

\begin{proof}
Fix $i\in \{1,\ldots,m\}$ and set $X:=X_i$ and $Y:=Y_i$. Let $k = |X|$. We decompose elements $\bar d$ from $[0,1]^{\binom{n}{2}}$ as $\bar d=(\bar d',\bar d^{k+1},\ldots,\bar d^{n},\bar d'')$, where $\bar d'\in [0,1]^{\binom{k}{2}}$, $\bar d''\in [0,1]^{n-\binom{k}{2}}$, and $\bar d^t\in [0,1]^k$ for $t=k+1,\ldots,n$.  The intention is that $\bar d'$ represents $d_{ij}$ for $1\leq i<j\leq k$, $\bar d^t$ represents $d_{it}$ for $i=1,\ldots,k$ and $t=k+1,\ldots,n$, and $\bar d''$ represents $d_{ij}$ for $k < i<j\leq n$.

Let $E_n$ denote the projection of $D_n$ onto the last $\binom{n-k}{2}$ coordinates, and set $S_n:=[\frac{1}{2},1]^{\binom{k}{2}} \times [\frac{1}{2} + \delta_n, 1]^{k(n-k)} \times E_n$. Note that the values of $d$ specified by an element in $S_n$ cannot violate the triangle inequality, so we have $S_n\subseteq D_n$.

Let $B_n$ be the set of $\bar{d} \in D_n$ such that:
\begin{itemize}
    \item $\Conf_X(\bar d')\leq \epsilon$, and
    \item $\Conf_Y(\bar d',\bar d^t)> \epsilon \text{ for all }t=k+1,\ldots,n$.
\end{itemize}
If we let $A = \{\bar{d}'\in M_k :  \Conf_X(\bar{d}') \leq \epsilon\}$, then we have $\lambda_n(B_n)\leq \lambda_k(A)\cdot ((\frac{1}{2}+\delta_n)^k - \epsilon^k)^{n-k}\cdot \lambda_{n-k}(E_n)$, so
\begin{align*}
    \mu_n(B_n) &=\frac{\lambda_n(B_n)}{\lambda_n(D_n)}\leq \frac{\lambda_n(B_n)}{\lambda_n(S_n)}\\
    &\leq \frac{\lambda_k(A)\cdot ((\frac{1}{2}+\delta_n)^k-\epsilon^k)^{n-k}\cdot \lambda_{n-k}(E_n)}{(\frac{1}{2})^{\binom{k}{2}}\cdot (\frac{1}{2}-\delta_n)^{k(n-k)}\cdot \lambda_{n-k}(E_n)}\\
    &= 2^{\binom{k}{2}}\lambda_k(A) \left(\frac{(\frac{1}{2} + \delta_n)^k - \epsilon^k}{(\frac{1}{2} - \delta_n)^k}\right)^{n-k}
\end{align*}

We have that
\[
\frac{(\frac{1}{2} + \delta_n)^k - \epsilon^k}{(\frac{1}{2} - \delta_n)^k} \leq 
 \left ( \protect{\frac{\frac{1}{2} + \delta_n}{\frac{1}{2} - \delta_n}} \right )^k 
- (2\epsilon)^k
\]
and so, since $\lim_{n\to \infty}\delta_n=0$, it follows that there exists a constant $C$ and a constant $p<1$ such that $\mu_n(A_n)\leq Cp^n$ for all large enough $n$.

The previous calculation yielded an upper bound on the probability that a random element of $D_n$ failed the extension axiom $\Psi^\epsilon_{X_i\sqsubset Y_i}$ as witnessed by the first $k$ elements.  The calculation is identical if one focuses on any other $k$ element subset instead of the first $k$ coordinates.  Moreover, if we were considering $m$ extension axioms instead of just one, we would obtain a similar expression, possibly with different constants $C$ and $p$.  It follows that for some constants $K$ and $q < 1$,
\[
\mu_n\left(\left\{\bar d\in D_n \ : \ \max_{i=1,\ldots,m}\Psi_{X_i\sqsubset Y_i}^\epsilon(\bar d)>0\right\}\right)\leq mn^kKq^n
\]
for sufficiently large $n$.  As $n$ tends to infinity, this quantity goes to zero, yielding the desired result.
\end{proof}

\begin{cor}\label{newkeycor}
For any $X_1 \sqsubset Y_1, \ldots, X_m \sqsubset Y_m$ from $\mathcal C$ and any $\epsilon > 0$, we have
\[
\lim_{n\to \infty}\nu_n\left(\left\{\bar d\in M_n \ : \ \max_{i=1,\ldots,n}\Psi_{X_i\sqsubset Y_i}^\epsilon(\bar d)=0\right\}\right)=1.
\]
\end{cor}

\begin{proof}
Let $A:=\left\{\bar d\in M_n \ : \ \max_{i=1,\ldots,n}\Psi_{X_i\sqsubset Y_i}^\epsilon(\bar d)=0\right\}$.  We then have
$$\nu_n(A)=\mu_n(A\cap D_n)\cdot \nu_n(D_n)+\nu_n(A\setminus D_n).$$ By Theorem \ref{newkeytheorem}, $\lim_{n\to \infty}\mu_n(A\cap D_n)=1$.  Since $\lim_{n\to \infty}\nu_n(D_n)=1$, it follows that $\lim_{n\to \infty}\nu_n(A)=1$, as desired. 
\end{proof}

\section{The approximate zero-one law}

We recall the compactness theorem for continuous logic.  Given a theory $T$ (in some language), define the theory $T^+$ to consist of all sentences $\sigma\dminus \epsilon$, where $\sigma\in T$ and $\epsilon>0$.  We say that $T$ is \emph{approximately finitely satisfiable} if $T^+$ is finitely satisfiable.  
\begin{fact}[Compactness Theorem]\label{compactness}
Given a theory $T$, we have that $T$ is satisfiable if and only if it is approximately finitely satisfiable.
\end{fact}

We let $T_{\as}$ denote the $L$-theory consisting of the set of all extension axioms $\psi^{\epsilon}_{X\sqsubset Y}$ together with the sentence $\thesentence$:
\[
\sup_x\sup_y (\min\{d(x,y),\frac{1}{2}\dminus d(x,y)\})
\]
which when satisfied says that $d(x,y) \geq \frac{1}{2}$ when $x\neq y$.
A familiar amalgamation construction shows that $T_{\as}$ is satisfiable.  By Corollary \ref{newkeycor}, we can say more.

\begin{prop}
$T_{\as}$ has the \emph{finite model property}:  every finite subset of $T_{\as}$ is approximately satisfied in a finite metric space.  Equivalently, there is an ultraproduct of finite metric spaces which satisfies $T_{\as}$.
\end{prop}

\begin{proof}
By Fact~\ref{CS} and Corollary \ref{newkeycor}, any finite number of extension axioms, together with any sentence $\thesentence\dminus \epsilon$  for $\epsilon >0 $, are satisfied in a sufficiently large finite metric space, whence the proposition follows.  The ultraproduct equivalence is a standard reformulation of the finite model property. (See, for example, \cite{GL}.)
\end{proof}

In order to prove our zero-one law for finite metric spaces, we show that $T_{\as}$ is a complete theory, that is, for all models $X$ and $Y$ of $T_{\as}$, we have that $X$ and $Y$ are elementarily equivalent.  We establish this fact using \emph{Ehrenfeucht-Fra\"iss\' e games}.

\begin{defn}
Given metric spaces $X$ and $Y$, $n\in \mathbb N$, and $\epsilon>0$, we define $\mathfrak G(X,Y,n,\epsilon)$ to be the two-player, $n$-round game, where at round $i$, player I chooses $a_i\in X$ or $b_i\in Y$ and then player II chooses $b_i\in Y$ or $a_i\in X$ accordingly.  We say that player II wins a run of $\mathfrak G(X,Y,n,\epsilon)$ if $$|d_X(a_i,a_j)-d_Y(b_i,b_j)|<\epsilon$$ for all $1\leq i<j\leq n$; otherwise, player I wins.  We write $X\equiv_{n,\epsilon}Y$ if player II has a winning strategy in $\mathfrak G(X,Y,n,\epsilon)$.
\end{defn}

The classical version of the following fact is well-known (see, e.g. \cite[Theorem 2.4.6]{dave}).  For a proof (of a slight variant) in the continuous setting, one can consult \cite[Lemma 2.4]{McDuff}.

\begin{fact}
If $X$ and $Y$ are metric spaces, then $X$ and $Y$ are elementarily equivalent if and only if $X\equiv_{n,\epsilon}Y$ for all $n\in \mathbb N$ and $\epsilon>0$.
\end{fact}

\begin{thm}
$T_{\as}$ is complete.
\end{thm}

\begin{proof}
Fix $X,Y\models T_{\as}$.  It suffices to show that $X\equiv_{n,\epsilon}Y$ for all $n\in \mathbb N$ and $\epsilon>0$.  We prove this by induction on $n$, the base case $n=1$ being trivial.  Now suppose that $n>1$ and $X\equiv_{n-1,\epsilon} Y$ for all $\epsilon>0$.  Fix $\epsilon>0$ and let player II play the first $n-1$ rounds of $\mathfrak G(X,Y,n,\epsilon)$ according to a winning strategy for $\mathfrak G(X,Y,n-1,\frac{\epsilon}{2})$, yielding $a_1,\ldots,a_{n-1}\in X$ and $b_1,\ldots,b_{n-1}\in Y$ with $$|d_X(a_i,a_j)-d_Y(b_i,b_j)|<\frac{\epsilon}{2}$$ for all $1\leq i<j\leq n-1$.  Now suppose that player I plays $a_n\in X$ in the final round of $\mathfrak G(X,Y,n,\epsilon)$ (the case that they play $b_n\in Y$ is handled in a symmetric fashion).  Let $X_0:=\{a_1,\ldots,a_{n-1}\}$ and let $Y_0:=X_0\cup \{a_n\}$. Then $X_0$ and $Y_0$ are in $\mathcal{C}$, since $Y\models \thesentence$. Then $Y\models \psi^{\frac{\epsilon}{2}}_{X_0\sqsubset Y_0}$ and $ \Conf_{X_0}(b_1,\ldots,b_{n-1})<\frac{\epsilon}{2}$, so we have that $$Y\models \inf_w \left(\Conf(b_1,\ldots,b_{n-1},w)\dminus \frac{\epsilon}{2}\right),$$ whence there is $b_n\in Y$ such that $\Conf_{Y_0}(b_1,\ldots,b_n)<\epsilon$.  It follows that this strategy is winning for player II in $\mathfrak G(X,Y,n,\epsilon)$.
\end{proof}

Given a sentence $\sigma$ in the language of metric spaces, let $\sigma^{\as}$ denote the unique real number $r$ such that $\sigma^X=r$ for all $X\models T_{\as}$.  It follows that, given any $\epsilon>0$, the theory $T_{\as}\cup \{\epsilon\dminus |\sigma-\sigma^{\as}|\}$ is not satisfiable.  By Fact \ref{compactness}, there are extension axioms $\psi^{\epsilon_i}_{X_i\sqsubset Y_i}$ ($1\leq i\leq m$) and $\eta>0$ such that the theory
$$\left\{\thesentence\dminus \eta,\left(\max_{1\leq i\leq m}\psi^{\epsilon_i}_{X_i\sqsubset Y_i}\right)\dminus \eta,\epsilon\dminus |\sigma-\sigma^{\as}|\right\}$$ is not satisfiable.  Combining this observation with Fact~\ref{CS} and Corollary \ref{newkeycor} immediately yields:

\begin{thm}[Approximate zero-one law]\label{thm:main}
For any sentence $\sigma$ in the language of metric spaces and any $\epsilon>0$, we have
$$\lim_{n\to \infty}\nu_n(\{X\in M_n \ : \ |\sigma^{X}-\sigma^{\as}|<\epsilon\})=1.$$
\end{thm}

\section{Further model-theoretic properties of \texorpdfstring{$T_{\as}$}{T\_AS}}

In this last section, we gather some further model-theoretic facts about $T_{\as}$.  We assume that the reader is familiar with the model-theoretic terms defined in this section.

\begin{thm}\label{QE}
$T_{\as}$ has quantifier-elimination and is the model-completion of the theory 
$T_0:=\{\thesentence\}$.
\end{thm}

\begin{proof}
By a standard model-theoretic test for quantifier-elimination (see \cite[Proposition 13.6]{mtfms}), it is enough to prove the following:  given $X,Y\models T_{\as}$, $A\subseteq X$, an isometric embedding $f:A\hookrightarrow Y$, and $a\in X\setminus A$, there is an elementary extension $Y\preceq Y'$ and an isometric embedding $g:A\cup\{a\}\hookrightarrow Y'$ extending $f$.  However, this follows easily from Fact \ref{compactness} and the fact that $Y\models T_{\as}$.

In order to prove that $T_{\as}$ is the model-completion of $T_0$, it remains to show that every model of $T_0$ embeds in a model of $T_{\as}$. By Fact~\ref{compactness}, it suffices to show that, for any $X_0=\{a_1,\ldots,a_n\}$ from $\mathcal{C}$ and $\epsilon>0$, there is $X\models T_{\as}$ and $b_1,\ldots,b_n\in X$ such that $\Conf_{X_0}(b_1,\ldots,b_n)<\epsilon.$  We do this by induction on $n$, the case $n=1$ being trivial. Suppose that $n>1$ and the claim is true for $n-1$.  Fix $X_0$ as above and $\epsilon>0$.  By induction, there is $X\models T_{\as}$ and $b_1,\ldots,b_{n-1}\in X$ such that, setting $X_0':=\{a_1,\ldots,a_{n-1}\}$, we have $\Conf_{X_0'}(b_1,\ldots,b_{n-1})<\frac{\epsilon}{2}$.  Since $X\models \psi^{\frac{\epsilon}{2}}_{X_0'\sqsubset X_0}$, it follows that there is $b_n\in X$ such that $\Conf_{X_0}(b_1,\ldots,b_n)$, as desired. 
\end{proof}

\begin{cor}
$T_{\as}$ has continuum many non-isomorphic separable models.
\end{cor}

\begin{proof}
Any model of $T_{\as}$ is topologically discrete, so any separable model of $T_{\as}$ is countable. In any such model, the metric $d$ only takes on countably many values in $[0,1]$. By Theorem \ref{QE}, given any separable model $X_0$ of $T_0$, there is a separable model $X$ of $T_{\as}$ such that $X_0$ embeds into $X$, so for every value $r\in [\frac{1}{2},1]$, there is a separable model $X$ of $T_{\as}$ such that the metric takes on the value $r$ in $X$. It follows that $T_{\as}$ has continuum many non-isomorphic separable models.  
\end{proof}

\begin{thm}
$T_{\as}$ is not stable but is supersimple of $U$-rank $1$.  Moreover, forking independence is characterized by
$$A\ind_C B\Leftrightarrow A\cap B\subseteq C,$$ where $A,B,C$ are small subsets of some monster model $\mathbb X$ of $T_{\as}$. 
\end{thm}

\begin{proof}
It is straightforward to verify that the independence relation in the above display satisfies all of the axioms of forking independence in simple theories, whence we can conclude that $T_{\as}$ is simple and the above independence relation is forking independence.\footnote{See \cite[Theorem 5.4.5]{jeanmartin}.}  We verify only the Independence Theorem over Models.  Suppose that $X$ is a model, $X\subseteq A$, $X\subseteq B$, $A\ind_X B$, and $p(x)\in S(A)$ and $q(x)\in S(B)$ are independent extensions of their common restriction $p_0\in S(X)$.  Since $p$ and $q$ are independent extensions of $p_0$, all distances specified by $p$ and $q$ between $x$ and elements of $(A\cup B)\setminus X$ are at least $\frac{1}{2}$, whence we may find an abstract extension $Y:= A\cup B\cup \{a\}$ of $A\cup B$ such that $Y\models T_0$ and  $a$ is a tuple realizing the quantifier-free parts of both $p$ and $q$.  Since $T_{\as}$ is the model completion of $T_0$ and $\mathbb X$ is saturated and strongly homogeneous, we may  embed $Y$ in $\mathbb X$ over $A\cup B$. By quantifier elimination, the image of $a$ in $\mathbb{X}$ satisfies both $p$ and $q$. Consequently, the Independence Theorem over Models holds.

To see that the $U$-rank of the theory is $1$, suppose that $p\in S_1(A)$ is a type with $U(p)\geq 1$.  Take a forking extension $q\in S_1(B)$ of $p$ and let $a\models p$.  Then $a\in B\setminus A$ and thus the condition $d(x,a)=0$ belongs to $q$.  It follows that $q$ is algebraic, whence $U(p)=1$.  Recall also that theories of $U$-rank $1$ are supersimple.  \footnote{See \cite[Definition 1.19 and Proposition 1.20]{supersimple} for a discussion of these matters in the setting of compact abstract theories, a precursor to modern continuous model theory.}

To see that $T_{\as}$ is not stable, let $p(x)$ be any $1$-type over a model $X$, let $a\in \mathbb X$ realize $p$ and take $b\in \mathbb X\setminus Xa$.  Then we can assign $d(x,b)$ to be any number in $[\frac{1}{2},1]$ and obtain an extension of $p$ to $Xb$ in this manner.  Thus, there are continuum many different nonforking extensions of $p$ to $Xb$, whence $T_{\as}$ is not stable by \cite[Theorem 14.12]{mtfms}. 
\end{proof}

The reader should contrast the previous result with the theory of the Urysohn sphere, which is not simple (see \cite[Theorem 5.4]{EG}).

\end{document}